\documentclass{article}

\usepackage{amsfonts, amsthm, amsmath, amssymb}
\usepackage{mathrsfs}
\usepackage[english]{babel}
\usepackage{graphicx}

\usepackage{babelbib}

\theoremstyle{plain}
  \newtheorem{theorem}{Theorem}[section]
  
  \newtheorem{lemma}[theorem]{Lemma}
  \newtheorem{corollary}[theorem]{Corollary}
\theoremstyle{definition}

\theoremstyle{remark}

\newcommand{\comma}{\textrm{,}}

\newcommand{\period}{\textrm{.}}
\newcommand{\bra}[1]{\left( #1 \right)}

\newcommand{\sqa}[1]{\left[ #1 \right]}
\newcommand{\cur}[1]{\left\{ #1 \right\}}

\newcommand{\abs}[1]{\left| #1 \right|}

\newcommand{\rnum}{\mathbb{R}}

\newcommand{\Prob}{\mathbb{P}}

\newcommand{\cA}{\mathcal{A}}

\newcommand{\cF}{\mathcal{F}}

\newcommand{\sgn}{\operatorname{sgn}}
\newcommand{\veps}{\varepsilon}

\begin{document}

\title{Zero noise limits using local times}
\author{Dario Trevisan\footnote{ \small{Scuola Normale Superiore, Piazza dei Cavalieri, 7, 56126, Pisa, Italy }. Email: \emph{dario.trevisan@sns.it}.}}




\maketitle

\begin{abstract}
We consider a well-known family of SDEs with irregular drifts and the correspondent zero noise limits. Using (mollified) local times, we show which trajectories are selected. The approach is completely probabilistic and relies on elementary stochastic calculus only.
\end{abstract}


\section{Introduction and results}

For fixed $\gamma \in [0,1)$, let us consider the following ODE, in integral form:
\[ x\bra{t} = \int_0^t\sgn(x\bra{s}) \abs{x\bra{s}}^\gamma ds \quad t\ge0\]
where $\sgn\bra{x} = I_{\cur{x>0}} - I_{\cur{x<0}}$. It is well known that there are infinitely many solutions and they are all of the form $\pm H_\gamma\bra{t-t_0}$, with
\[ \quad H_\gamma \bra{s}= [(1-\gamma) s^+]^{\frac {1}{1-\gamma}}  \comma\]
for some $t_0 \ge 0$.

Given $\veps>0$, let us consider a small random perturbation of the ODE above:

\begin{equation}\label{sde}\left\{ \begin{array}{l} d X^\veps_t = \sgn(X^\veps_t) \abs{X^\veps_t}^\gamma dt+ \veps dW_t \\ X^\veps_0 = 0 \period \end{array}\right. \end{equation}
Weak existence and uniqueness in law are then guaranteed by Girsanov theorem and Novikov condition: indeed, the random varibale $K \int_0^T  \abs{W_s}^{2\gamma}$ is exponentially integrable for any $T,K>0$. Therefore one can consider a weak solution $X^\veps$ defined on  some space $\bra{\Omega, \cA, \Prob, \bra{\cF}_{0\le t \le T}, \bra{W_t}_{0\le t \le T}} $ and let $\mu^\veps = \mu_\gamma^\veps$ be the law of $X^\veps$, which is a probability measure on the Borel sets of $C_0[0,T]$, equivalent to the Wiener measure.

It is also known that the family of measures $\bra{\mu_\gamma^\veps}_{0<\veps\le1}$ is tight: indeed, in the case $\gamma = 0$, it follows from the estimate, with $C = 1$,
\[ \abs{X^\veps_t -X^\veps_s} \le  C\abs{t-s} + \veps \abs{W_t - W_s}\period \]
In general, estimating the drift term $\abs{x}^\gamma \le 1 + \abs{x}$, one obtain that with arbitrary high probability $X^\veps_t$ is uniformly bounded in the interval $[0,T]$, and so the estimate above holds, for some $C>0$. By tightness, therefore, one can consider some sequence $\veps_n \to 0$ such that $\mu^{\veps_n}$ weakly converges to some probability measure $\mu$ (depending on $\gamma$).

To characterize $\mu$ is the prototype of zero noise problems, which appear in many contexts: for brevity, here we refer the extended overview presented in Chapter 1, Section 5 in \cite{flandoli}. Among the results already there, we remark that the work \cite{attanasio} discusses a zero noise limit for some linear PDEs of transport type related to the family of ODEs introduced above. For more recent developments, not included in \cite{flandoli}, we mention the forthcoming article \cite{delarue}, where a two-dimensional zero noise problem with discontinuous drift is solved; \cite{borkar} and \cite{buckdahn}, where zero noise problems for perturbed ODEs with respectively continuous and measurable drifts are discussed.

However, in the context of perturbed one-dimensional SDEs, the most general results are still those in \cite{bafico}, which rely on explicit estimates for exit times, obtained by solving related PDEs. The results obtained there show that, in the particular case introduced above, $\mu$ is concentrated only on the trajectories $\pm H_\gamma$, which leave immediately the origin.

The aim of this short paper is to provide an entirely probabilistic proof of a concentration result for $\mu$, in this special case as strong as that obtainable by applying the methods in \cite{bafico}, but relying only on applications of It\^o(-Tanaka) formula and elementary estimates for stochastic integrals. In fact, local times and Tanaka formula appear only in the proof for the special case $\gamma = 0$, but the general case is a technical development of the simple idea exploited there, after a suitable mollification procedure.

Before stating the main results, we remark the family of examples introduced above is well studied in the literature and much more can already be said about the limit probability $\mu$. In \cite{gradinaru} and \cite{hermann}, large deviations estimates are proved, by computing explicitly the density of $X^\veps_t$ and expanding it in terms of eigenfunctions of a Sch\"rodinger operator (there are currently many efforts to extend the classical Wentzell-Freidlin large deviations theory in the case of irregular coefficients: see e.g.\ \cite{caramellino} and the monograph \cite{feng}). Another approach is presented in \cite{pamen}, where a general setting for small noise problems is introduced, using Malliavin calculus both to prove strong existence and compactness of families of strong solutions for the SDE. We remark that computations involving It\^o-Tanaka formulas and local times appear also in these works, but they are not used to investigate the concentration properties of $\mu$. In the proof of Proposition 3 in \cite{gradinaru}, local times appear when manipulating the expression provided by Girsanov theorem, while in \cite{pamen} they appear in Example 2.11, in the expression for the Malliavin derivative of a solution $X^\veps$.

The concentration result for $\mu$ that we are going to prove is indeed a corollary of the following theorem.

\begin{theorem}\label{theo:01}
Given $T>0$, there exist positive numbers $\bar{t}, h, \alpha$, depending only on $\gamma, T, \veps$, infinitesimal as $\veps \to 0$ (and the other parameters are fixed) such that, given any weak solution $\bra{X^\veps_t}_{0\le t\le T}$ of \eqref{sde}, with probability greater than $1-\alpha$, it holds
\begin{itemize}
\item either for any $\bar{t}\le t \le T$, $\abs{X^\veps_t - H^\gamma\bra{t}} \le h$
\item or for any $\bar{t}\le t \le T$, $\abs{X^\veps_t + H^\gamma\bra{t}} \le h$.
\end{itemize}
\end{theorem}

The main feature of this result, together with its proof, is that it provides a rigorous deduction of the following intuition, which is not evident at all in the classical approach in \cite{bafico}: as $\veps \to 0$, the trajectory $X^\veps\bra{\omega}$ is forced by the noise to follow closely one of the two extremal trajectories $\pm H^\gamma$, and this selection happens in a small time interval $[0,\bar{t}]$. Moreover, the quantities $\bar{t}, h, \alpha$ can be computed explicitly.

We deduce immediately the following existence and characterization result for the zero noise limit probability $\mu_\gamma$.

\begin{corollary}\label{coro:01}
The weak limit $\mu = \lim_{\veps \to 0}  \mu^\veps$ exists and is given by  \begin{equation}\label{law} \mu =\frac 1 2 \delta_{H^\gamma} +\frac 1 2 \delta_{-H^\gamma}\period\end{equation}
\end{corollary}

\begin{proof}
Since $\bra{\mu^\epsilon}_{\epsilon>0}$ is tight, it is enough to consider a convergent subsequence $\mu^{\epsilon_n}$ and prove that its limit is given by the expression above.

For fixed $t,\eta >0$, it holds
\[ \lim_{\veps \to 0} \mu^{\epsilon} \cur{ \omega \in C_0[0,T] \,\, : \, \abs{ \abs{\omega\bra{t}} - H_\gamma \bra{t}  } > \eta }  \le \alpha \comma \]
whenever $\veps$ is small enough so that $\bar{t} \le T$ and $h \le \eta$, where $\bar{t}, h, \alpha$ are those provided by the theorem above.

By lower semicontinuity of weak convergence of measures on open sets, it holds therefore
\[ \mu\bra{ \omega \in C_0[0,T] \,\, : \, \abs{ \abs{\omega\bra{t}} - H_\gamma \bra{t}  } > \eta } = 0 \comma\]
which entails that $\mu$ is a probability measure concentrated at most on $\pm H_\gamma$, being $t, \eta$ arbitrary.

The simmetry of the problem allows us to conclude that $\mu$ is given by \eqref{law}, since every $\mu^{\epsilon}$ is  invariant under the transformation $\omega \mapsto -\omega$.
\end{proof}

\section{Proof of Theorem \ref{theo:01}}\label{section2}

\subsection{Case $\gamma = 0$}


Given $\veps>0$ and a weak solution $X^\epsilon$, we write It\^o-Tanaka formula for the local time at $0$, with respect to the semimartingale $X^\veps$ (Theorem 1.2, Chapter VI in \cite{revuz}), i.e.\
\[ \abs{X^\veps_t} =  \int_0^t \sqa{\sgn\bra{X^\veps_s}}^2 ds+ \veps \int_0^t \sgn\bra{X^\veps_s}  dW_s + L^0\sqa{X^\veps}_t \comma \]
for any $t\ge0$. Since $\sgn\bra{x}^2 = I_{\cur{x\neq0}}$ and the local time process $L^0\sqa{X^\veps}_t$ is non negative, we obtain
\[ \abs{X^\veps_t} \ge \int_0^t I_{\cur{X^\veps_s \neq 0}} ds + \veps \int_0^t \sgn\bra{X^\veps_s} dW_s \period \]
As already remarked, by Girsanov theorem, the law of $\bra{X^\veps}_{0\le t \le T}$ is equivalent to the Wiener measure and therefore
\[ X^\veps_t \bra{\omega} \neq 0\quad \textrm{ $\Prob \otimes \mathscr{L}$-a.e.\ $(\omega,t) \in \Omega \otimes [0,T]$, } \]
where $\mathscr{L}$ is the Lebesgue measure on the interval: indeed the same holds true for a Wiener process in place of $X^\veps$. It follows that, almost surely, for $0\le t \le T$, $\int_0^t I_{\cur{X^\veps_s \neq 0} }ds = t$ and therefore
\begin{equation} \label{eq-gamma-0-min} \abs{X^\veps_t} \ge t - \veps \abs{  \int_0^t \sgn\bra{X^\veps_s} dW_s} \period \end{equation}
On the other hand, for any $t\ge0$, directly from \eqref{sde} written in integral form, we deduce that
\begin{equation}\label{eq-gamma-0-max} \abs{X^\veps_t} \le t + \veps \abs{  W_t } \period\end{equation}


The estimates \eqref{eq-gamma-0-min} and \eqref{eq-gamma-0-max} above imply that, given $\eta >0$, the event
\[ \cur{ \sup_{0\le t \le T} \big| \abs{X^\veps_t} - t \big| > \eta }\]  is contained in the union

\[\cur{ \sup_{0\le t \le T} \big| \int_0^t \sgn\bra{X^\veps_s} dW_s \big|> \frac{\eta}{\veps}} \cup \cur{ \sup_{0\le t \le T} \abs{W_t} > \frac{\eta} {\veps}}\period\]
On the other hand, Doob's inequality and It\^o's isometry assure that
\[ \Prob\bra{ \sup_{0\le t \le T} \big| \int_0^t \sgn\bra{X^\veps_s} dW_s \big|> \frac{\eta} {\veps}}+  \Prob\bra{ \sup_{0\le t \le T} \abs{W_t} > \frac{\eta} {\veps}} \le 2\veps^2 \frac{T}{\eta^2} \period \]

In order to compute $\bar{t}, h, \alpha$ as required by the theorem, we fix any $a$, with $0<a<1$ and put $\eta = \veps^a$ above, so that, with probability greater than $1-\alpha$, where $\alpha = 2\veps^{2(1-a)} T$, it holds
\[ \sup_{0\le t \le T} \abs{ \abs{X^\veps_t} - t } \le \veps^a \period\]

Then, we put $h =  \veps^a$ and $\bar{t} = 2 h $ so that, in the event above, it holds for any $\bar{t} \le t \le T$,
\[ \abs{X^\veps_t}  \ge t - h \ge \bar{t} - h  = h > 0 \period \]
and in particular $X^\veps_t$ does not change sign. Therefore,
\begin{itemize}
\item either for any $\bar{t}\le t \le T$, $\abs{X^\veps_t} = X^\veps_t$ and $\abs{ X^\veps_t - t } \le h$,
\item or  for any $\bar{t}\le t \le T$, $\abs{X^\veps_t} = -X^\veps_t$ and $\abs{ X^\veps_t + t } \le h$.
\end{itemize} 

\subsection{Case $\gamma \in (0,1)$}

The main difficulty in this case is due to the fact that the drift term is infinitesimal in zero. Indeed, if we repeat the same passages as above, we obtain that
\begin{equation} \abs{X^\veps_t} =  \int_0^t \abs{X^\veps_s}^\gamma ds+ \veps \int_0^t \sgn\bra{X^\veps_s} dW_s + L^0 \sqa{X^\veps}_t \period \end{equation}
If we simply drop the local time we cannot conclude that $\abs{X^\veps_t}$ grows enough and the solution leaves the origin. We are going to see that the local time term indeed contains exactly the information that we need to conclude that the solution with high probability moves away from zero. This, combined with the fact that the drift drags the solution away from zero in a finite time, will lead to the conclusion.

To extract easily this information from the local time term, we mollify the map $x \mapsto \abs{x}$ and define
\[ x \mapsto \abs{x}_\delta = \frac{1}{\delta} \int_{-\delta}^{\delta} \rho\bra{y / \delta} \abs{x -y} dy \comma\]
where $\rho\bra{x} \in C^\infty\bra{\rnum}$ is non negative, supported in $[-1,1]$ with $\rho\bra{x}\ge 3/4$ for $x \in [-\frac 1 2, \frac 1 2]$ and $\int_{-1}^1 \rho\bra{x} dx = 1$.

The positive map defined in this way is smooth and for any $x \in \rnum$, it holds 
\[ \abs{ \abs{ x} - \abs{x}_\delta} \le \delta \textrm{, } \abs{x}_\delta' \le 1 \textrm{ and } \abs{x}_\delta'' \ge 0 \period\] Moreover, the assumption $\rho\bra{x}\ge 3/4$ for $x \in [-\frac 1 2, \frac 1 2]$ entails that, when $\abs{x}\ge \delta/2$,  it holds $\abs{x}_\delta' \sgn\bra{x} \ge 1/2$, while for $\abs{x}\ge \delta$, it holds $\abs{x}_\delta' \sgn\bra{x} \ge 1$. Finally, for $\abs{x} \le \delta/2$, $\abs{x}_\delta'' \ge 3 / 4 \delta$.

 
We apply It\^o formula to $\abs{X^\veps_t}_\delta$ so that, for $t\ge0$, it holds
\begin{equation}\label{ito}\abs{X^\veps_t}_\delta = \abs{0}_\delta + \int_0^t\abs{X^\veps_s}_\delta' \sgn\bra{X^\veps_s} \abs{X^\veps_s}^\gamma + \veps\int_0^t\abs{X^\veps_s}_\delta'dW_s + \frac  1 2 \veps^2 \int_0^t\abs{X^\veps_s}_\delta'' ds \period\end{equation}

Let us state a key estimate in form of a lemma. 

\begin{lemma}\label{lemma1}
Fix $a >2\gamma/\bra{1+\gamma}$. Then there are positive numbers $\bar{t}$, $\delta$, depending only on $\gamma,a,\veps$, infinitesimal as $\veps \to 0$ (and the other parameters are fixed) such that, a.s.\ on the event
\begin{equation}\label{event} \cur{\sup_{0\le t \le T} \veps \abs{\int_0^t \abs{X^\veps_s}_\delta'dW_s} \le  \veps^a}\comma \end{equation}
for every $t$, with $\bar{t} \le t \le T$, it holds $\abs{X^\veps_t}_\delta\ge 2\delta +\veps^a$ and  $\abs{X^\veps_t}\ge \delta$.
\end{lemma}


\begin{proof} Using the estimates for $\abs{\cdot}_\delta$ and its derivatives, from \eqref{ito} we obtain that
\[  \abs{X^\veps_t}_\delta \ge \int_0^t \sqa{ I_{\cur{\abs{X^\veps_s}\ge \delta/2} }\frac {\delta^\gamma}{2^{\gamma+1}} + I_{\cur{\abs{X^\veps_s} < \delta/2}} \bra{\frac{3\veps^2}{8\delta} - \frac{ \delta^\gamma}{2^{\gamma}} } }ds + \veps\int_0^t\abs{X^\veps_s}_\delta'dW_s  \period \]

We put $\delta = c_1 \veps^{2/\bra{1+\gamma}}$, where $c_1= 2^{(\gamma- 2) /(1+\gamma)}$ is a positive number depending only on $\gamma$, such that
\[ \bra{\frac{3\veps^2}{8\delta} - \frac{ \delta^\gamma}{2^{\gamma}} } = \frac{\delta^\gamma}{2^{\gamma+1}}\period \]
Thanks to this choice, for any $t\ge0$, it holds
\[ \abs{X^\veps_t}_\delta \ge \frac{ \delta^\gamma}{2^{\gamma+1}} t - \veps \abs{ \int_0^t\abs{X^\veps_s}_\delta'dW_s} \period \]
In the event \eqref{event}, it entails that, for $0\le t \le T$,
\[ \abs{X^\veps_t}_\delta \ge \frac{ \delta^\gamma }{2^{\gamma+1}} t - \veps^{a}\period \]

We put $\bar{t} = 2^{\gamma+1} \bra{2 \delta + 2\veps^{a}}/\delta^{\gamma}$, which  is immediately seen to be infinitesimal as $\veps \to 0$, since $\gamma < 1$ and $a > 2\gamma/(1+\gamma)$. With this choice the inequality above entails that, for any $\bar{t}\le t \le T$,
\[ \abs{X^\veps_t}_\delta\ge 2 \delta + \epsilon^a \comma\]
which leads to the thesis, since $\abs{X_t^\epsilon} \ge \abs{X_t^\epsilon}_\delta - \delta$.
\end{proof}

In order to conclude the proof of Theorem \ref{theo:01}, let us fix $a \in ]2\gamma/(1+\gamma), 1[$, so that the lemma just proved provides some $\bar{t}$, $\delta$. 

Applying It\^o's formula to $\abs{X^\veps_t}_\delta$, starting from $\bar{t}$, we obtain
\[ \abs{X^\veps_t}_\delta \ge \abs{X^\veps_{\bar{t}}}_\delta  +  \int_{\bar{t}}^t\abs{X^\veps_s}_\delta' \sgn\bra{X^\veps_s} \abs{X^\veps_s}^\gamma -\veps\abs{\int_{\bar{t}}^t\abs{X^\veps_s}_\delta'dW_s} \comma\]
since $\abs{x}_\delta'' \ge 0$. On the event \eqref{event}, it holds therefore, for $\bar{t} \le t \le T$,
\[  \abs{X^\veps_t} \ge  \delta  +  \int_{\bar{t}}^t\abs{X^\veps_s}_\delta' \sgn\bra{X^\veps_s} \abs{X^\veps_s}^\gamma ds\comma \]
where we used the fact that $\abs{\abs{x}_\delta - \abs{x}} \le \delta$ and the estimate on $\abs{X^\epsilon_{\bar{t}}}_\delta$ provided by the lemma. But the lemma shows that also $\abs{X^\epsilon_t}\ge \delta$, so that, as already remarked, $\abs{X^\epsilon_t}_\delta' \sgn\bra{X^\epsilon_t} \ge 1$ and therefore for $\bar{t} \le t \le T$,
\[  \abs{X^\veps_t} \ge \delta  +\int_{\bar{t}}^t\abs{X^\veps_s}^\gamma ds \period\]

Lemma \ref{lemma2} below allows us to conclude that, in the event \eqref{event}, a.s.\ it holds, for $\bar{t}\le t \le T$,
\[ \abs{ X_t^\veps } \ge H^\gamma \bra{t-R\bra{\bar{t},\delta} }\period\]
where $R\bra{\bar{t},\delta}$ is some (explicit) quantity depending also on $\gamma$, which is infinitesimal as $\veps \to 0$, so that the r.h.s.\ above converges uniformly in $t\in[0,T]$ to $H^\gamma \bra{t}$ as $\epsilon \to 0$.

On the other hand, directly from \eqref{sde}, we obtain the estimate
\[ \abs{X^\veps_t} \le \int_0^t \abs{X^\veps_s}^\gamma ds + \veps \abs{W_t}\comma \]
that, thanks to another application of Lemma  \ref{lemma2}, entails that on the event
\begin{equation}\label{eq-event2} \cur{ \sup_{0\le t \le T} \veps \abs{W_t} \le \veps^a} \end{equation}
it holds
\[ \abs{ X_t^\veps } \le H^\gamma  \bra{t-R\bra{0,\veps^a}} \period\]
Again, as $\veps \to 0$, $R\bra{0,\veps^a} \to 0$ and therefore the r.h.s.\ above converges uniformly in $t\in[0,T]$  to $H^\gamma \bra{t}$.

By applying Doob's inequality and It\^o's isometry, we see that the intersection of the events \eqref{event} and \eqref{eq-event2} has probability greater than $1-\alpha$, where we introduce the infinitesimal $\alpha = 2T \veps^{2-a}$. On this intersection, it holds, for $\bar{t} \le t\le T$,
\[ \abs{ \abs{ X_t^\veps } - H\bra{t} } \le h \comma \]
where
\[ h = \max_{0\le t \le T}  \cur{ \abs{H^\gamma\bra{t} - H^\gamma \bra{t-R\bra{\bar{t},\delta}}}, \abs{H^\gamma\bra{t} - H^\gamma \bra{t-R\bra{0,\veps^a}} }} \]
which is easily seen to be infinitesimal, as $\veps \to 0$.

On the other hand, we already know from Lemma \ref{lemma1} above that $\abs{X_t^\veps } \ge \delta \neq 0$, for $\bar{t} \le t \le T$ and therefore we can conclude as in the proof for the $\gamma =0$ case. Indeed, in the intersection of the events \eqref{event} and \eqref{eq-event2}, whose probability is greater than $1-\alpha$, almost surely,
\begin{itemize}
\item either for any $\bar{t}\le t \le T$, $\abs{X^\veps_t} = X^\veps_t$ and 
$\abs{ X^\veps_t -   H^\gamma\bra{t} } \le h$,
\item or  for any $\bar{t}\le t \le T$, $\abs{X^\veps_t} = -X^\veps_t$ and 
$\abs{ X^\veps_t + H^\gamma\bra{t} } \le h$,
\end{itemize} 

and the proof is completed. It remains only the following comparison lemma that was used above.

\begin{lemma}\label{lemma2}
Given $0\le \bar{t} \le T$, $\delta >0$ and a non-negative continuous function $f$ on $[\bar{t}, T]$, such that for any $t\in[\bar{t}, T]$,
\[ f\bra{t} \ge \delta + \int_{\bar{t}}^t f\bra{s}^\gamma ds \quad  \textrm{(respectively, $\le$)}\period \]
Then, for $t\in[\bar{t}, T]$,
\[ f\bra{t} \ge H^\gamma\bra{ t - R\bra{\bar{t}, \delta}} \quad \textrm{(respectively, $\le$)}\period\]
where $R\bra{\bar{t}, \delta} = \bar{t} - \delta^{1+\gamma}/(1-\gamma)$.
\end{lemma}

\begin{proof}

The term $R\bra{\bar{t}, \delta}$ is defined in such a way that $t \mapsto H^\gamma\bra{ t - R}$ is the solution of the ODE, 
\[ x\bra{t} = \delta + \int_{\bar{t}}^t x\bra{s}^\gamma ds\comma \quad (\bar{t} \le t \le T)\period\]
Moreover, $H^\gamma\bra{ t - R}$ is continuously differentiable on $[\bar{t}, T]$.

Let $R = R\bra{\bar{t}, \delta}$ and set
\[ D\bra{t} = \delta + \int_{\bar{t}}^t f\bra{s}^\gamma ds - H^\gamma\bra{ t - R}\comma \]
which is continuously differentiable on $[\bar{t}, T]$, with $D\bra{\bar{t}}=0$ and derivative
\[ D'\bra{t} = f\bra{t}^\gamma - \bra{H^\gamma}'\bra{ t - R} = f\bra{t}^\gamma - \bra{H^\gamma
\bra{ t - R}}^\gamma\]
by the remark above. Now, since $x \mapsto x^\gamma$ is increasing for $x>0$, from the hypothesis we obtain that for $t\in[\bar{t},T]$, the condition $D\bra{t} \ge 0$ implies $D'\bra{t} \ge 0$ and so we conclude  that $D\bra{t}\ge0$ for $t$ in this range. The other case is similar.
\end{proof}

\section*{Acknowledgements}
The author is grateful to M.~Maurelli for many discussions on the subject and thanks F.~Flandoli and M.~Pratelli for their support.

\bibliographystyle{amsplain}
\bibliography{biblio}

\providecommand{\bysame}{\leavevmode\hbox to3em{\hrulefill}\thinspace}
\providecommand{\MR}{\relax\ifhmode\unskip\space\fi MR }
\providecommand{\MRhref}[2]{%
  \href{http://www.ams.org/mathscinet-getitem?mr=#1}{#2}
}
\providecommand{\href}[2]{#2}
\begin{thebibliography}{10}

\bibitem{attanasio}
Stefano Attanasio and Franco Flandoli, \emph{Zero-noise solutions of linear
  transport equations without uniqueness: an example}, C. R. Math. Acad. Sci.
  Paris \textbf{347} (2009), no.~13-14, 753--756. \MR{2543977 (2010i:60172)}

\bibitem{bafico}
R.~Bafico and P.~Baldi, \emph{Small random perturbations of {P}eano phenomena},
  Stochastics \textbf{6} (1981/82), no.~3-4, 279--292. \MR{665404 (83j:60082)}

\bibitem{caramellino}
P.~Baldi and L.~Caramellino, \emph{General {F}reidlin-{W}entzell large
  deviations and positive diffusions}, Statist. Probab. Lett. \textbf{81}
  (2011), no.~8, 1218--1229. \MR{2803766 (2012g:60086)}

\bibitem{borkar}
V.~S. Borkar and K.~Suresh Kumar, \emph{A new {M}arkov selection procedure for
  degenerate diffusions}, J. Theoret. Probab. \textbf{23} (2010), no.~3,
  729--747. \MR{2679954 (2011m:60171)}

\bibitem{buckdahn}
R.~Buckdahn, Y.~Ouknine, and M.~Quincampoix, \emph{On limiting values of
  stochastic differential equations with small noise intensity tending to
  zero}, Bull. Sci. Math. \textbf{133} (2009), no.~3, 229--237. \MR{2512827
  (2010b:60167)}

\bibitem{delarue}
F.~Delarue, F.~Flandoli, and D.~Vincenzi, \emph{Noise prevents collapse of
  {V}lasov-{P}oisson point charges}, {To appear in } Comm. Pure Appl. Math.

\bibitem{feng}
J.~Feng and T.G. Kurtz, \emph{Large deviations for stochastic processes},
  Mathematical Surveys and Monographs, vol. 131, AMS, 2006.

\bibitem{flandoli}
Franco Flandoli, \emph{Random perturbation of {PDE}s and fluid dynamic models},
  Lecture Notes in Mathematics, vol. 2015, Springer, Heidelberg, 2011, Lectures
  from the 40th Probability Summer School held in Saint-Flour, 2010.
  \MR{2796837 (2012c:60162)}

\bibitem{gradinaru}
Mihai Gradinaru, Samuel Herrmann, and Bernard Roynette, \emph{A singular large
  deviations phenomenon}, Ann. Inst. H. Poincar\'e Probab. Statist. \textbf{37}
  (2001), no.~5, 555--580. \MR{1851715 (2002j:60042)}

\bibitem{hermann}
Samuel Herrmann, \emph{Ph\'enom\`ene de {P}eano et grandes d\'eviations}, C. R.
  Acad. Sci. Paris S\'er. I Math. \textbf{332} (2001), no.~11, 1019--1024.
  \MR{1838131 (2002c:60098)}

\bibitem{pamen}
Olivier~Pamen Menoukeu, Thilo Meyer-Brandis, and Frank~Norbert Proske, \emph{A
  {G}el'fand triple approach to the small noise problem for discontinuous
  {ODE}'s},  (2010).

\bibitem{revuz}
Daniel Revuz and Marc Yor, \emph{Continuous martingales and {B}rownian motion},
  third ed., Grundlehren der Mathematischen Wissenschaften [Fundamental
  Principles of Mathematical Sciences], vol. 293, Springer-Verlag, Berlin,
  1999. \MR{1725357 (2000h:60050)}

\end{thebibliography}

\end{document}